\documentclass[a4paper,11pt]{amsart}

\usepackage[T1]{fontenc}

\usepackage[top=1.4in, bottom=0.8in, left=0.8in, right=0.8in]{geometry}

\newtheorem{theo}{Theorem}[section]
\newtheorem{lem}{Lemma}[section]

\theoremstyle{definition}
\newtheorem*{defi*}{Definition}

\theoremstyle{plain}
\newtheorem{thmint}{Theorem}

\newcommand\defeq{\,\mathrel{\stackrel{\makebox[0pt]{\mbox{\normalfont\tiny def}}}{=}}\;}

\def\Cinf{\mathcal{C}^\infty}
\DeclareMathOperator\Lie{Lie}
\DeclareMathOperator{\im}{i}
\DeclareMathOperator{\imm}{Im}
\DeclareMathOperator\id{id}

\let\txtaccH\H  \let\txtaccS\S  \let\txtaccL\L  \let\txtacct\t

\renewcommand{\H}{\ifmmode\mathcal H\else\expandafter\txtaccH\fi}
\renewcommand{\L}{\ifmmode\mathcal L\else\expandafter\txtaccL\fi}
\renewcommand{\S}{\ifmmode\mathbb S\else\expandafter\txtaccS\fi}
\renewcommand{\t}{\ifmmode\mathfrak t\else\expandafter\txtacct\fi}

\def \A {\mathcal A}
\def \U {\mathcal U}
\def \X {\mathfrak X}
\def \F{\mathcal{F}}
\def \R {\mathbb R}
\def \C {\mathbb C}
\def \V {\mathcal V}
\def \W {\mathcal W}

\def \g {\mathfrak g}
\def \h {\mathfrak h}

\newcommand{\p}{\partial}
\renewcommand{\bar}{\overline}

\title{Pluriclosed deformations of Bismut flat metrics}

\author{Giuseppe Barbaro}
\address{Giuseppe Barbaro\newline
		\textsc{\indent Instituto de Ciencias Matem\'aticas (ICMAT), CSIC-UAM - UC3M - UCM\newline 
			\indent Nicol\'as Cabrera 13--15, 28049 Madrid, Spain}}
\email{giuseppe.barbaro@icmat.es}

\thanks{This work was supported by the European Union’s Horizon Europe research and innovation programme under the Marie Skłodowska-Curie Actions (Grant agreement No. 101273232 - SURF FLOW; HORIZON-MSCA-2025-PF).
The author has also been supported by a DFF Sapere Aude grant ``Conformal geometry: metrics and cohomology". The author is a member of GNSAGA of INdAM.}

\begin{document}
	
\begin{abstract}
    We deform the complex structure of compact simply-connected Bismut flat manifolds, obtaining new complex manifolds that admit pluriclosed metrics but do not admit Bismut flat metrics.
    This produces genuinely new examples of pluriclosed manifolds while enriching our understanding of the deformation theory of pluriclosed metrics.
    Moreover, our main result is relevant for the search of pluriclosed solitons and non-trivial Bismut Hermitian Einstein metrics in all dimensions.
\end{abstract}
	
\maketitle
	
\section{Introduction}\label{sec: intro}
The notion of a {\em pluriclosed metric} was introduced by Bismut in the context of non-K\"ahler index theory \cite{MR1006380}, and is closely related to the {\em Bismut connection} -- the unique Hermitian connection whose torsion tensor is totally skew-symmetric. On a Hermitian manifold $(M,J,g)$, the torsion $3$-form of this connection is given by $Jd\omega$, where $\omega= g(J\cdot,\cdot)$ denotes the fundamental form of the metric. A Hermitian metric is called {\em pluriclosed}
if its fundamental form satisfies $dJd\omega=0$, i.e. if the torsion $3$-form of the Bismut connection is closed. Beyond index theory, the interest in pluriclosed metrics stems from several sources. In mathematical physics, they arise naturally in supersymmetric sigma models and related theories \cite{MR851702, MR776369, MR800347, MR872720}. In complex geometry, they play a distinguished role because every conformal class of Hermitian metrics on a compact complex surface contains a pluriclosed representative \cite{MR470920}; this makes them a central tool in the geometrization of compact complex surfaces via the {\em pluriclosed flow} \cite{MR2673720, MR3110582, MR4181011}. Finally, they are one of the fundamental ingredients of {\em generalized K\"ahler geometry} \cite{MR776369, MR3232003}.

Beyond complex surfaces, we are only beginning to understand which non-K\"ahler manifolds can support pluriclosed metrics, with relatively few examples known and construction methods beginning to emerge.
A fundamental family is provided by compact Lie groups endowed with a bi-invariant metric and a compatible {\em Samelson} complex structure \cite{MR59287, MR1836272, zhao2020strominger}. These Hermitian structures are precisely the {\em Bismut flat} ones, namely those whose Bismut connection has vanishing curvature \cite{MR4127891}. Further existence results include toric bundles over Hermitian manifolds \cite{MR2406264, MR2736170, MR4554058}, certain instanton moduli spaces \cite{MR2314216}, and products of Sasakian manifolds \cite{MR4733370}. The existence problem has also been studied systematically on nilmanifolds and solvmanifolds, where it has led to a number of existence and classification results \cite{MR4480232, MR2059435, MR2926995, MR4462315, MR4891832, MR4329266, MR2797819, MR4838566}. Finally, under the additional assumption that the torsion of the Bismut connection is parallel, pluriclosed manifolds admit a complete classification \cite{MR5124088}.
Remarkably, most of these examples arise on manifolds carrying a complex toric fibration over a Hermitian manifold. 
This common geometric feature suggests that toric foliations provide a natural framework for the existence of pluriclosed metrics. A further class of examples, coming from the recent construction of {\em Bismut Hermitian--Einstein manifolds} \cite{ALL,MR5008125}, is again realized in this way, as a toric foliation over a K\"ahler orbifold; since these examples are relevant to the present work, we shall return to them in greater detail below.

The homogeneous case has been investigated systematically for Wang`s {\em C-spaces}, which are compact complex manifolds with finite fundamental group admitting a transitive action by a compact Lie group of biholomorphisms. It was shown in \cite{MR4032184} that such a manifold admits a pluriclosed metric if and only if it is the product of a compact Lie group and a K\"ahler homogeneous C-space.
Then, it was proved in \cite{MR4592898} that on an even-dimensional compact simply connected Lie group endowed with a Samelson complex structure, the existence of a pluriclosed metric forces the complex structure to be compatible with a bi-invariant metric, so that the resulting Hermitian manifold is necessarily Bismut flat.
As a byproduct, since the compatibility with a bi-invariant metric cuts out a proper linear subspace of the space of Samelson structures, we have a simple illustration of the fact that the existence of pluriclosed metrics is not an open property under deformations of the complex structure.
The general deformation theory of pluriclosed metrics was developed by Cavalcanti using the tools of generalized geometry \cite{cavalcanti13, MR4157573}, who showed that the obstructions to deforming a pluriclosed metric live in the Dolbeault cohomology group $H^{2,1}_{\bar\p}$. In agreement with the discussion above, this group is never trivial for Bismut flat manifolds, see e.g. \cite{barbaro2024}.

In the present work we construct explicit pluriclosed deformations of Bismut flat structures. More precisely, starting from a Bismut flat Hermitian manifold, we deform its Samelson complex structure to a non-invariant complex structure that still carries pluriclosed metrics. Since Bismut flat structures are necessarily homogeneous, the deformed complex structures cannot admit Bismut flat metrics, and therefore yield genuinely new examples of pluriclosed manifolds.

\begin{thmint}\label{th: main}
    Let $(G,J,g)$ be a compact simply-connected Lie group equipped with a Bismut flat Hermitian structure. Then, there exists an open deformation space $0\in\U\subset\R^{\ell}$ and a smooth deformation $\{J_t\}_{t\in\U}$ of the complex structure $J(=J_0)$ such that
    \begin{itemize}
        \item[i.] the set of $t\in\U$ such that the complex structure $J_t$ is not biholomorphic to any left-invariant complex structure on $G$ is dense in $\U$;
        \item[ii.] there is a family of metrics $\{g_t\}_{t\in\U}$ varying smoothly on $t$ such that $(J_t,g_t)$ is a pluriclosed structure for $t\in\U$ and $g_0=g$.
    \end{itemize}
\end{thmint}

The strategy of the proof rests on the toric foliation structure discussed above. More precisely, via the {\em Tits fibration}  $G\rightarrow G/T$, we regard $G$ as a toric fibration over a K\"ahler C-space $G/T$, where $T$ is the maximal torus detected by the Samelson complex structure $J$.
In other words, this provides $G$ with the structure of a transversely K\"ahler holomorphic foliation, and it is this foliated structure that we deform. 
We then run a Kodaira--Spencer type argument in which the role played by {\em Bott--Chern cohomology} in the K\"ahler setting is taken over by {\em Aeppli cohomology} in the pluriclosed one. 
A careful analysis of the cohomology of the deformed spaces is possible exploiting the existence of a cohomological model for transversely K\"ahler holomorphic toric foliations.
We thus show that the dimension of the relevant Aeppli cohomology group remains constant along the deformation, and this constancy implies the existence of a smooth projection of the Bismut flat metric to a pluriclosed metric on each deformed space.

Theorem \ref{th: main} shows that, although pluriclosed metrics are not stable under arbitrary deformations, their existence is nevertheless preserved along deformations that respect the peculiar geometry of Bismut flat spaces. It is worth stressing that, as recalled above, holomorphic toric foliations are by no means rare in pluriclosed geometry. 
We will indeed notice that the deformations in Theorem \ref{th: main} are precisely built on the deformations of the transversal holomorphic structure on $G$.
Therefore, it would be natural to further study on which conditions the existence of a pluriclosed metric on the total space of an initial toric holomorphic foliation persists under deformations of this structure.

In this article, we focus on the deformation of Bismut flat metrics because these are related to the search for ``non-trivial" {\em Bismut Hermitian--Einstein metrics}. Before explaining this link, let us recall a few crucial things on this special geometry.
First of all, a pluriclosed metric is called {Bismut Hermitian--Einstein} (BHE) if its Bismut Ricci form vanishes (cf. \cite{MR4629758, MR4284898}). 
Remarkably, the BHE equation can be interpreted as a Hermitian--Einstein equation for a twisted holomorphic bundle \cite{MR4629758}, relating the existence of these metrics to an algebro-geometric stability condition.
The interest in BHE metrics then comes from string theory as well as their link with the pluriclosed flow.
On one hand, the metric and the torsion three-form satisfy a supergravity equation arising from the string effective action \cite{Ivanov_2001,Polchinski,MR3110582}.
On the other, they are static points of the pluriclosed flow, being therefore crucial in understanding its long-time behavior and its use as a geometrization tool.
 
For a relatively long time Bismut flat metrics were the only known compact BHE examples. 
This scarcity is not accidental, indeed, while this is a natural non-K\"ahler counterpart of the relatively vast K\"ahler Calabi--Yau geometry, in the non-K\"ahler world, twisting the Einstein equation with a non-trivial torsion makes it remarkably more rigid.
A number of results in this direction show that, under mild additional assumptions, BHE metrics are forced to be Bismut flat.
In the homogeneous setting, it was proved in \cite{barbaro2024} that no non-flat BHE metric can exist on C-spaces.
Moreover, Gauduchon--Ivanov proved that the only non-K\"ahler BHE surfaces are the finite quotients of the standard Bismut flat Hopf surface \cite{MR1477631}; while, in arbitrary dimension, the structure of compact non-K\"ahler BHE manifolds was described in \cite{MR5008125}: such manifolds carry two commuting Killing fields of constant length, and hence a rank-$2$ toric foliation with a transversely balanced structure satisfying further analytic conditions. In complex dimension $3$ this reduces the BHE equation to a scalar equation on a K\"ahler surface.

Overall, whether non-trivial BHE metrics -- that is, metrics which are not locally Bismut flat -- exist at all on compact manifolds remained an open question for quite some time, with other partial results and contributions by \cite{MR4883225, MR5084158, MR5124088, MR5067254}.
This question was settled very recently in \cite{ALL}, where the first non-trivial examples were produced in complex dimension $3$. 
Working from the reduction in \cite{MR5008125}, they study the reduced equation on a K\"ahler orbifold surface where the powerful machinery of K\"ahler geometry applies. It is worth observing that some of the examples of \cite{ALL} are deformations of the Bismut flat structure on the Calabi--Eckmann threefold. Seen from this angle, their examples not only motivate the result in Theorem \ref{th: main}, but suggest a natural strategy for producing non-trivial BHE metrics in arbitrary dimension through a continuity argument starting from a Bismut flat structure.
In relation to this, we shall also notice that a compact Bismut flat manifold cannot admit any non-trivial BHE metric \footnote{This is a result obtained together with F. Pediconi that will happear soon.}.
Therefore, for such an argument to have any chance of getting off the ground, one first needs the deformed complex structures to carry pluriclosed metrics, on spaces admitting no Bismut flat metric and, crucially, one needs these metrics to depend smoothly on the parameter and to converge to the flat one. This is precisely what Theorem \ref{th: main} guarantees.

We end this introduction with a few further comments on this larger project. Being Theorem \ref{th: main} established, the step that follows is the analysis of the linearized Bismut Ricci form at a Bismut flat metric. In the K\"ahler setting this is exactly the mechanism underlying LeBrun--Simanca's deformation theory of constant scalar curvature metrics \cite{MR1274118}.
We shall highlight that while a LeBrun--Simanca type openness result is already available on the reduced K\"ahler side \cite{ALL}, what is missing is the corresponding statement upstairs, for deformations of the complex structure.
Here the obstruction space should not be expected to vanish in general, since Bismut flat manifolds have large symmetry groups. For example, the Hopf surfaces already show what could happen: there, the analogue of the deformation of Theorem \ref{th: main} does not produce BHE metrics; what one finds instead are {\em pluriclosed solitons} described in \cite{MR4023384,MR4287690, ZZsolitons}.
Determining whether the deformations of Theorem \ref{th: main} carry genuine BHE metrics, solitons, or neither in higher dimensions, and how the answer is governed by the geometry of the underlying toric foliation, is the question we intend to study next.

\medskip
\noindent {\it Acknowledgments.} The author is grateful to Vestislav Apostolov for insightful discussions.

\section{Preliminaries}\label{sec: prelim}
This section is devoted to the introduction of the notions that will be essential for the proof of the main result. 
Since the strategy of the proof is to regard our space as a foliation and deform this structure, while simultaneously deforming all transverse structures in a compatible way, we begin by describing the framework of abelian foliations. Within this setting, we briefly discuss the geometry of Bismut flat manifolds. Moreover, we introduce the appropriate cohomological tools. 
In particular, we recall the definition of the Aeppli cohomology and construct a family of elliptic operators adapted to the deformation of the foliation, designed to compute the corresponding basic cohomology along the deformation.

\subsection{Abelian Foliations}
Let $M$ be a compact manifold of dimension $m$ and denote by $\X(M)$ the infinite-dimensional Lie algebra of vector fields on $M$. 
Assume that $\t \subset \X(M)$ is an abelian $r$-dimensional subalgebra, with $r \leq m$, which is {\it free}, namely the rank of the evaluation map
$$
{\rm ev}_x : \t \to T_xM \,\, , \quad {\rm ev}_x(V) \defeq V_x
$$
is maximal for any $x \in M$. 
We call this data an abelian foliation on $M$, and we write $\F=(M,\t)$.
Notice that $\t$ gives rise to a regular $r$-dimensional distribution $\V \subset TM$ by setting $\V_x \defeq {\rm ev}_x(\t)$, which is called {\it vertical distribution}. 
By the Frobenius Theorem, $\V$ is integrable, and each maximal leaf of $\V$ is called {\it $\t$-orbit}. We remark that, although $M$ is locally equivalent to a principal bundle, the orbit space $M/\t$ can be highly non-smooth. 

A (possibly vector-valued) tensor field $\Phi$ on $M$ is said to be: {\it $\t$-invariant} if $\L_V\Phi =0$ for any $V \in \t$; {\it horizontal} if $\iota_V \Phi = 0$ for any $V \in \t$; {\it basic} if it is both $\t$-invariant and horizontal. 
Notice that the exterior differential $df$ of a $\t$-invariant function $f$ is basic. Moreover, the exterior differential $d\alpha$ of a basic form $\alpha$ is basic, and so this allows to define the {\it basic de Rham cohomology spaces} $H^\bullet_b(M,\F)$ (see, e.g., \cite[Section 7.1]{MR1362865} and references therein).

A smooth $\t$-valued $1$-form $\mu : TM \to \t$ is called {\it principal connection} if it is $\t $-invariant and verifies ${\rm ev}_x(\mu_x(V_x)) = V_x$ for any $V \in \t $, $x \in M$ \cite[Section 3.1]{MR1362865}. Consequently, the kernel $\H \defeq {\rm ker}(\mu) \subset TM$ is a regular $(m-r)$-dimensional distribution that is transverse to $\V$, which is called {\it horizontal distribution}. 
When a horizontal distribution is fixed, a vector field $X \in \X(M)$ is called: {\it horizontal} if $X_x \in \mathcal{H}_x$ for any $x \in M$; {\it basic} if it is both $\t $-invariant and horizontal.
It is straightforward to check that the curvature $2$-form $\Omega\in H^2(M,\t)$ defined as
\begin{equation*} \label{eq:defOmega}
\Omega \defeq d\mu
\end{equation*}
is basic. 
Moreover, by \cite[Proposition 7.5]{MR1362865}, the basic cohomology class $[\Omega] \in H^2_b(M,\F,\t)$ does not depend on the principal connection $\mu$, and is called {\it characteristic class of $\F$.}

\medskip

We now describe holomorphic and metric structures on $M$ which are compatible with the foliation structure.
\begin{defi*}
    Let $\F=(M,\t)$ be an abelian foliation, and $J$ be a complex structure on $M$. Suppose that $J$ is such that
    \begin{itemize}
        \item $\t$ {\it preserves $J$}, namely $\L_VJ=0$ for any $V \in \t$, or else $\t$ is the Lie algebra of a subgroup of the identity component of the group of biholomorphisms of $(M,J)$;
        \item $\t$ is {\it $J$-invariant}, namely $J\t = \t$.
    \end{itemize}
    then we call $(\F,J)$ a {\em holomorphic foliation} on $M$.
\end{defi*}

\noindent In this paper we will mostly work with structures which are defined on the normal bundle of a foliation, i.e. that are transverse to the leaves. We define them as follows.
\begin{defi*}
    Let $\F=(M,\t)$ be an abelian foliation, with connection $\mu$ detecting the horizontal complement $\H$.  
    A tensor $J\in\mathrm{Aut}(\H)$ is a {\em transverse complex structure} of $\F$ if
    \begin{itemize}
    \item $J^{2} = - \id$, 
    \item $\mathcal{L}_{Z}J=0$ for every $Z\in\t$ and
    \item the Nijenhaus tensor $N_{J}(X,Y)=[JX,JY]-J[JX,Y]-J[X,JY]-[X,Y]$ of $J$ vanishes on any horizontal vector fields $X$ and $Y$.
    \end{itemize}
    We call $(\F,\H,J)$ a {\em transversely holomorphic foliation} of $M$.
\end{defi*}
\begin{defi*}
     Let $\F=(M,\t)$ be an abelian foliation, with connection $\mu$ detecting the horizontal complement $\H$. A tensor $\check g \in \Cinf \left(\H^*\otimes\H^*\right)$ is a {\em transverse metric} of $\F$ if
    \begin{itemize}
    \item $\check{g}$ is symmetric, positive definite and
    \item $\L_{Z}\check{g}=0$ for every $Z\in\t$.
    \end{itemize}
    We call $(\F,\H,\check g)$ a {\em transversely Riemannian foliation} of $M$.
\end{defi*}

\noindent Merging these structures we obtain a {\em transversely Hermitian} structure. However, we directly give here the definition of {\em transversely K\"ahler} structure.
\begin{defi*}
    The data $(\F,\H,J,\check g)$ is a {\em transversely K\"ahler foliation} of $M$ if
    $(\F,\H,J)$ is a transversely holomorphic foliation of $M$, $(\F,\H,\check{g})$ is a transversely Riemannian foliation of $M$, and the tensor field $\check\omega$ defined by $\check\omega(X,Y) = \check{g}(X,JY)$ is antisymmetric and closed when regarded as a $2$-form on $M$ by the injection $\bigwedge^{2} \H^* \to \bigwedge^{2} T^{*}M$.
\end{defi*}

\noindent Notice that we can make sense of transverse structures even without fixing a connection for the foliation. However, we chose to define them in this way to simplify the exposition since for the purpose of this paper we will always be in the condition of having a preferred horizontal distribution.

\subsection{Transversal Laplacian}\label{subs: new Laplacian}
The main difficulty in working with the basic cohomology of a family of foliations is that the space of basic forms does not vary smoothly. 
The way of dealing with this fenomena is to consider horizontal forms and add the Lee derivative to the Laplacian in order to obtain harmonic forms which still are basic.
This idea derives from a construction El Kacimi--Alaoui and Hector used to prove the transverse Hodge decomposition for Riemannian foliations \cite{MR865667}, and was already successfully used to study the basic cohomology of Sasakian manifolds in \cite{MR3523249, MR4278212}. The material in this subsection is an adaptation of these latter works to our case, where the foliation structure is detected by an abelian subalgebra $\t\subset\X(M)$ instead of the Reeb vector field.

\medskip

Consider a transversely Riemannian foliation $(\F,\H,\check{g})$ on $M$ compact, with fibers of dimension $r$ and total space of dimension $\dim M=m=n+r$.
Fix a basis $(V_1,{\dots},V_{r})$ for $\t $ and consider the corresponding splitting of $\mu$:
$$
\mu = \mu_1 \otimes V_1 + {\dots} + \mu_{2k} \otimes V_{r} \,\, .
$$
We can define a double grading of the differential forms 
$\A^k(M)= \bigoplus_{i+j=k} \A_{i,j}(M,\F)$ where
$$\A_{i,j}(M,\F) \defeq \Cinf \left( \bigwedge^i \H^* \otimes \bigwedge^j \V^*\right).$$
Let us denote the space of horizontal $k$-form as 
$$\A_\H^k(M,\F)\defeq \Cinf\left(\bigwedge^k \H^*\right).$$
Then the space of basic $k$-form coincide with the subspace 
$$\A_b^k(M,\F)= \left\{ \alpha\in \A_\H^k(M,\F) \;|\; \L_V\alpha=0\;\forall V\in\t\right\}.$$
While the differential of a basic form is itself basic, this is not true for horizontal forms. Therefore, we have in a natural way a basic differential operator $d_b:\A_b^k(M,\F)\rightarrow\A_b^{k+1}(M,\F)$ that extends to $d_\H:\A_\H^k(M,\F)\rightarrow\A_\H^{k+1}(M,\F)$ as $d_\H=\pi_{\A_\H^{k+1}}\circ d$.
Now let $\star_\H:\A_\H^k(M,\F)\rightarrow\A_\H^{n-k}(M,\F)$ be the transverse Hodge-star operator with respect to $\check{g}$.
Notice that, since the metric $\check{g}$ is $\t$-invariant, the transverse Hodge-star operator restricts to a basic Hodge-star operator $\star_b:\A_b^k(M,\F)\rightarrow\A_b^{n-k}(M,\F)$. Indeed, $\L_V\circ \star_\H = \star_\H\circ\L_V$ for any $V\in\t$.
Moreover, $\star_\H$ induces a scalar product on $\A_\H^k(M,\F)$ as 
$$\left<\alpha,\beta\right> = \int_M \mu_1\wedge\cdots\wedge\mu_r\wedge\alpha\wedge\star_\H\beta, \quad\text{for }\alpha,\beta\in \A_\H^k(M,\F).$$
\begin{lem}
    The adjoint operator $d_\H^*$ of $d_\H$ with respect to $\left<\cdot,\cdot\right>$ is given by $d_\H^*=-\star_\H d_\H\,\star_\H$.
\end{lem}
\begin{proof}
    Take $\alpha\in\A_\H^{k-1}(M,\F)$ and $\beta\in\A_\H^{k}(M,\F)$. Then $\star_\H\beta\in \A_\H^{n-k}(M,\F)$. 
    By counting the horizontal/vertical degrees it is evident that 
    \begin{align*}
        d(\mu_1\wedge\cdots\wedge\mu_r\wedge\alpha\wedge\star_\H\beta)
            &= (-1)^r\mu_1\wedge\cdots\wedge\mu_r\wedge d\alpha\wedge\star_\H\beta + (-1)^{r+k-1}\mu_1\wedge\cdots\wedge\mu_r\wedge\alpha\wedge d\star_\H\beta
    \end{align*}
    We now compute
    \begin{align*}
        \left<d_\H\alpha,\beta\right> 
            &= \int_M \mu_1\wedge\cdots\wedge\mu_r\wedge d_\H\alpha\wedge\star_\H\beta \\
            &= \int_M \mu_1\wedge\cdots\wedge\mu_r\wedge d\alpha\wedge\star_\H\beta \\
            &= \int_M (-1)^k\mu_1\wedge\cdots\wedge\mu_r\wedge\alpha\wedge d\star_\H\beta\\
            &= \int_M (-1)^k\mu_1\wedge\cdots\wedge\mu_r\wedge\alpha\wedge d_\H\star_\H\beta
            = - \left<\alpha,\star_\H d_\H\star_\H\beta\right>
    \end{align*}
    proving the thesis.
\end{proof}
\noindent The lemma implies that $d^*_\H$ restricts to $d^*_b\defeq -\star_b\,d_b\,\star_b$ on the basic forms. In turns, this implies that the transverse Laplacian $\Delta_\H \defeq d_\H^*d_\H+d_\H d_\H^*$ restricts to the basic Laplacian $\Delta_b\defeq d_b^*d_b+d_bd_b^*$ on basic forms.
Recall that being $(\F,\H,\check{g})$ a transversely Riemannian foliation on a compact manifold $M$, the basic cohomology is given by $\Delta_b$-harmonic forms, namely
$$\mathbf{H}^k_b(M,\F) \defeq \ker \{\Delta_b:\A_b^k(M,\F)\rightarrow \A_b^k(M,\F)\}\cong H^k_b(M,\F).$$
Note that $\Delta_\H$ is only transversely strongly elliptic. For our later application of a theorem of Kodaira--Spencer we need a strongly elliptic operator acting on the space
of sections of a smooth bundle. 
We define it by
\begin{equation}\label{eq: new Laplacian}
    D \defeq \sum_{i=1}^r\L_{V_i}\L_{V_i} - \Delta_\H ,
\end{equation}
and notice that $\L_V$ preserves $\A_\H^k(M,\F)$ for any $V\in\t$.

\begin{lem}
    The differential operator $D$ is strongly elliptic and self-adjoint with respect to $\left<\cdot,\cdot\right>$.
\end{lem}
\begin{proof}
    For the ellipticity see \cite[Lemma 4.11]{MR3523249} or \cite[Lemma 3.2]{MR4278212} (see also \cite[p. 224]{MR865667}). 
    The self-adjointness is also a straightforward adaptation of their arguments, which we report here.
    To show that $\L_V$ is skew-symmetric for any $V\in\t$, take $\alpha,\beta\in\A_\H^k(M,\F)$ and notice that
    \begin{align*}
        \mathcal{L}_V(\mu_1\wedge\cdots\wedge\mu_r\wedge\alpha\wedge\star_\H\beta) 
            &= \mu_1\wedge\cdots\wedge\mu_r\wedge\L_V\alpha\wedge\star_\H\beta + 
                \mu_1\wedge\cdots\wedge\mu_r\wedge\alpha\wedge\star_\H\L_V\beta.
    \end{align*}
    We also have
    $$\L_V(\mu_1\wedge\cdots\wedge\mu_r\wedge\alpha\wedge\star_\H\beta) = d\iota_V(\mu_1\wedge\cdots\wedge\mu_r\wedge\alpha\wedge\star_\H\beta).$$
    Therefore, after integrating the result follows by Stokes theorem. 
\end{proof}
\noindent The crucial property of the operator $D$ is that it depends smoothly on the data of abelian algebra $\t$ generating the foliation on $M$, and it is possible to recover the basic harmonic forms from it by taking the $\t$-invariant forms in $\ker D$. As a matter of fact, a $\t$-invariant form $\alpha$ in $\ker D$ satisfies
$$ 0 = D\alpha =  \sum_{i=1}^r\L_{V_i}\L_{V_i}\alpha - \Delta_\H \alpha = - \Delta_b \alpha .$$

\subsection{Bismut flat and flag manifolds}\label{subs: flat and flag} 
We now recall a few crucial aspects of the geometry of compact simply-connected Bismut flat manifolds $(G,J,g)$ and their cohomology.
Thanks to \cite[Theorem 1]{MR4127891} they are compact simply-connected Lie groups $G$ equipped with a bi-invariant metric $g$ and a compatible left-invariant complex structure $J$. 
By Milnor's Lemma \cite[Lemma 7.5]{MR425012}, the Lie group $G$ is semisimple and $(G,g)$ is isometric to the product of simple Lie groups $G=G^{(1)} \times \cdots \times G^{(s)}$ each equipped with a negative multiple of the {\em Cartan Killing form}.
Moreover, all the left-invariant complex structures $J$ on $G$ are given by the {\em Samelson construction} \cite{MR59287,MR994129}. 
More precisely, $J$ detects a maximal torus $T\subset G$ that verifies the following properties (set $\t=\Lie{T}$ and $\g=\Lie{G}$): 
\begin{itemize}
    \item[$a)$] $\t$ is $J$-invariant, that is, $J\t = \t$;
    \item[$b)$] $J$ is ${\rm ad}(\t)$-invariant, that is, $[V,JX] = J[V,X]$ for any $V \in \t$, $X \in \g$.
\end{itemize}
By means of the above conditions, it follows that $J$ also projects via the {\em Tits fibration} 
\begin{equation}\label{eq: Tits} \tag{T-fib}
    \pi:G\rightarrow G/T
\end{equation} 
onto a $G$-invariant complex structure on the {\em flag manifold} $G/T$, which we still call $J$.
The data of a maximal torus $T\subset G$, a linear complex structure on $\t$ and a $G$-invariant complex structure on $G/T$ completely characterize the left-invariant complex structure $J$.
Furthermore, an invariant complex structure on $G/T$ corresponds to a choice of a {\em system of positive roots} for the {\em Cartan decomposition} of $\g$ associated to $\t$.
Then, take a system of {\em simple} positive roots $\{\alpha_1,\ldots,\alpha_r\}\subset\t^*$ for $r$ the rank of $G$. Their coroots $\{V_1,\ldots,V_r\}\subset\t$ give a basis of $\t$.
Hence, the principal connection $\mu\in\A^1(G,\t)$ for \eqref{eq: Tits} associated to the bi-invariant metric $g$ (which is the same of the one associated to the negative of the Cartan Killing form) decomposes as
$$\mu = \sum_{j=1}^r V_j\otimes\mu_j, \quad\text{ with }\quad \mu_j\in\A^1(G) \quad\text{for }j=1,\ldots,r. $$
In the same way the curvature form decomposes in {\em fundamental weights} $\Omega_1,\ldots,\Omega_r\in \A^2(G/T)$.
It is known (see e.g. \cite{MR880184}) that the invariant K\"aher metrics on $(G/T,J)$ are all of the form
\begin{equation}\label{eq: flag kahler}
    \omega_K = \sum_{j=1}^rc_j\Omega_j, \quad\text{ with }\quad c_j>0 \quad\text{for }j=1,\ldots,r.
\end{equation}
This can be traced back to the fact that the cohomology of $G/T$ is generated by the unit and the classes of the $\Omega_j$'s \cite{MR51508,MR102800}.
Furthermore, the only relations in cohomology between them are those given by the polynomials $\R[\Omega_1,\ldots,\Omega_r]$ which are invariant with respect to the action of the {\em Weil group}.
In particular, the first non-trivial relations occur in degree $2$ and are
\begin{equation}\label{eq: vanishing Q-form}
    \sum_{j=1}^{r_i}\left[\Omega_j^{(i)}\wedge\Omega_j^{(i)}\right] = 0,
    \quad\text{for }i=1,\ldots,s,
\end{equation}
where $T=T^{(1)}\times\cdots\times T^{(s)}$ with $T^{(i)}\subset G^{(i)}$ maximal torus of dimension $r_i$ for $i=1,\ldots,s$, and we used the superscript $^{(i)}$ to distinguish the simple components.
Building on this description, in \cite{barbaro2024} the cohomology of compact simply-connected Bismut flat manifolds was computed using the Tarn\'e model \cite{MR1255937}.

\subsection{Aeppli cohomology}
Given a complex manifold $(M,J)$, the {\em Aeppli cohomology} and {\em Aeppli number} in bi-degree $(p,q)$ are defined as
$$H_A^{p,q}(M,J) \defeq \frac{\ker \,\big( \partial\overline{\partial}: \mathcal{A}^{p,q}\rightarrow \mathcal{A}^{p+1,q+1}\big)}{\imm\,\big(\partial:\mathcal{A}^{p-1,q}\rightarrow \mathcal{A}^{p,q}\big) + \imm\,\big(\overline{\partial}:\mathcal{A}^{p,q-1}\rightarrow \mathcal{A}^{p,q}\big)}, \quad h^{p,q}_A(M,J)\defeq\dim H_A^{p,q}(M,J) .$$
Since the Aeppli cohomology groups can be realized as the kernel of a fourth-order elliptic operator \cite{IWASAWA_cohomology}, the Aeppli numbers are upper semicontinuous under deformations of the complex structure.
It is moreover evident that a pluriclosed metric $\omega$ on $(M,J)$ gives a $(1,1)$-Aeppli class, $[\omega]\in H^{1,1}_A(M,J)$. 
It turns out that, on compact manifolds, this class cannot vanish \cite{barbaro2024,marouani2023,MR4157573}. Hence, the Aeppli cohomology provides an obstruction to the existence of pluriclosed metrics on compact complex manifolds.
Furthermore, it plays the same fundamental role for pluriclosed metrics as the Bott--Chern cohomology does for K\"ahler metrics within the framework of Kodaira--Spencer theory. 
\begin{theo}[Theorem 8.11 in \cite{cavalcanti13}]\label{theo: Cavalcanti}
    Let $(M,J,g)$ be a pluriclosed manifold. Consider $0\in\U\subset\R^\ell$ an open neighborhood of the origin and let $\{J_t\}_{t\in\U}$ be a smooth family of complex structures with $J_0=J$. If the Aeppli number $h^{1,1}_A(M,J)$ is constant on $\U$, then there is a (possibly smaller) open neighborhood $0\in\U'\subset\U$ and family of metrics $\{g_t\}_{t\in\U'}$ varying smoothly on $t$ such that $(J_t,g_t)$ is a pluriclosed structure for $t\in\U'$ and $g_0=g$.
\end{theo}

This makes the Aeppli cohomology a powerful tool in studying pluriclosed geometry. For example, the cone of pluriclosd metrics on a compact simply connected Bismut flat manifold $(G,J,g)$ has been described in \cite{barbaro2024} in terms of the Aeppli cohomology. More precisely, for any pluriclosed metric $\omega$ there exists a $J$-Hermitian bi-invariant (hence Bismut flat) metric $\omega_{BF}$ such that $[\omega]=[\omega_{BF}]$ in $H^{1,1}_A(G,J)$.
This description was obtained from the following result, which we will also need.
\begin{theo}[Theorem 3.1 in \cite{barbaro2024}]\label{th: h11 Bflat}
    Let $G$ be a semisimple Lie group equipped with a bi-invariant metric $g$ and a compatible Samelson complex structure $J$. Suppose that $(G,J)$ is irreducible, then 
    $$h_A^{1,1}(G,J)= 1.$$
\end{theo}
\noindent In the above, {\em irreducible} means that $(G,J)$ does not biholomorphically decompose as a product of complex manifolds. 
Notice that any Samelson complex structure $J$ on a semisimple Lie group $G$ splits as a product of irreducible ones $(G,J)\cong (G_1,J_1)\times\cdots\times(G_p,J_p)$ (see \cite{barbaro2024} for the details).

\section{Pluriclosed deformations}\label{s: main}
This section is devoted to proving our main theorem about the existence of pluriclosed deformations of Bismut flat structures. 
The idea is to emulate the fenomena happening in sasaki geometry. To give an explainatory example, we can consider the Calabi Eckman threefold $\S^3\times\S^3$ with its standard structure which is in particular Bismut flat. We are allowed to deform the Sasaki structure of each of the sphere by moving the Reeb vector field in a two dimensional abelian algebra. We then equip the product $\S^3\times\S^3$ with the new complex structure associated to the sasaki structures as in \cite{MR5124088}. This deformation proces breaks the symmetries of the complex structure, which is not invariant anymore, hence cannot admit a Bismut flat metric. Nonetheless, it still admitis a pluriclosed structure by \cite{MR5124088, MR4733370}. In the following proof, we generalize this proces by deforming, instead of the Reeb field, the whole foliation detected by the maximal torus associated to J.

\begin{proof}[Proof of Theorem \ref{th: main}]
    We suppose, without loss of generality, that $(G,J)$ is irreducible. Indeed, the following argument can be repeated on each irreducible component. 
    We recall that with this assumption it holds $h_A^{1,1}(G,J)=1$, see Theorem \ref{th: h11 Bflat}.
    For clarity, the proof is structured in several steps, in which we construct the smooth deformation $\{J_t\}_{t\in\U}$ of the complex structure $J$ satisfying (i), and we prove that $h_A^{1,1}(G,J_t)=1$.
    Then, the condition (ii) will follow from Theorem \ref{theo: Cavalcanti}.
    \medskip
    
    {\bf Step 0:} {\em Deforming the foliation.}\\
    By hypothesis, $G$ is a semisimple Lie group $G=G^{(1)} \times \cdots \times G^{(s)}$ and the complex structure $J$ is $G$-left-invariant, hence, given by the Samelson construction (see Section \ref{subs: flat and flag}).
    This identifies a maximal torus $T=T^{(1)}\times\cdots\times T^{(s)}\subset G$,  with $T^{(i)}\subset G^{(i)}$ maximal tori of dimension $r_i$. Set $\t\defeq\Lie{T}\leq\g\defeq\Lie{G}$ and $\t^{(i)}\defeq\Lie{T^{(i)}}\leq\g^{(i)}\defeq\Lie{G^{(i)}}$, for $i=1,\ldots,s$.
    We can look at $\t$ as a subset of $\X(G)$ giving a foliation structure on $G$ (this is actually the $T$-principal bundle structure given by the Tits fibration \eqref{eq: Tits}).
    We consider $\H$ the horizontal complement given by the metric $g$. The complex structure $J$ restricts to $\H$ giving a transversely holomorphic foliation. The transversely holomorphic complex structure $J_{|_\H}$ can be identified with the complex structure induced by $J$ on the generalized flag $G/T$ through the Tits fibration.
    We remark that by hypothesis the metric $g$ is bi-invariant and the complex structure $J$ is left-invariant and right-$T$-invariant. 
    Thus the automorphism group is $\mathrm{Aut}^0(G,J,g) \cong G^L\times T^R$, with Lie algebra $\mathfrak{aut}(G,J,g) = \g^L+\t^R$, where we use superscripts $^R$ and $^L$ to distinguish the right and left actions.
    Now consider the Lie algebra $\t^L+\t^R$, which is abelian since the right and left torus actions commute. 
    Notice that the distribution $\H$ and the complex structure $J_{|_\H}$ are $\t_t$-invariant, for any $r$-dimensional abelian Lie algebra $\t_t\subset\t^L+\t^R$, thus giving a principal connection and a transversely holomorphic complex structure for the $\t_t$-foliation on $G$.
    We therefore consider $\t_t = \t_t^{(1)}+\cdots+\t_t^{(s)}$ with $\t_0=\t=\t^L$ and $\t_t^{(i)}$ abelian $r_i$-dimensional subalgebra of $\t^{(s)L}+\t^{(s)R}$ varying smoothly on $t\in\R^{r_i^2}$ for $i=1,\ldots,s$.
    This gives a smooth family of abelian foliations $\F_t=(G,\t_t)$, which are  transversely holomorphic foliations $(\F_t,\H,J_{|_\H})$ on $G$. 
    Here the moduli space $\R^{r_1^2+\cdots+r_s^2}$ is modeled on the Grasmanians $\mathsf{Gr}(r_1,2r_1)\times\cdots\times\mathsf{Gr}(r_s,2r_s)$.
    \medskip

    {\bf Step 1:} {\em Extending the complex structure.}\\
    We define the complex structure $J_t$ on $G$ by extending $J_{|_\H}$. To do so, we only need a linear complex structure on $\t_t$ varying smoothly on $t$. 
    As a matter of fact, the horizontal holomorphic distribution $\H^{1,0}$ is the same for any $J_t$ and it satisfies
    $$\left[\H^{1,0},\H^{1,0}\right]\subset\H^{1,0}$$
    thanks to the integrability of $J$ and the fact that the curvature form of the Tits fibration is of type $(1,1)$.
    Then, $[V,\H^{1,0}]=0$ for any $V\in \t_t$ since both the horizontal distribution and $J_{|_\H}$ are $\t_t$ invariant.
    Finaly, the fibers will be complex submanifolds.
    The precise definition of $J_{|\t_t}$ will be given later.
    
    \medskip
    
    {\bf Step 2:}{\em Proving (i).}\\
    Suppose that $J_t$ is left-invariant with respect to some group structure on $G$, then this group must be $G'$ isomorphic to $G$, and $J_t$ is a Samelson complex structure on $G'$. 
    Therefore, $J_t$ detects a maximal torus $H\subset G'$ with associated Tits fibration $\pi':G'\rightarrow G'/H$ on the flag manifold $G'/H$.
    Moreover, $J_t$ induces a complex structure on $G'/H$ (which we call $(J_t)_|$) that admits a K\"ahler metric $\omega_K$.
    Let us call $\g'$ and $\h$ the Lie algebras of $G'$ and $H$ respectively. Then, by construction, $\t_t$ is a relatively compact $J_t$-invariant abelian algebra of $J_t$-holomorphic vector fields, and $\t_t\subset\mathfrak{aut}(G',J_t)$. 
    We can consider the image of $\t_t$ through the projection $\rho:\mathfrak{aut}(G',J_t)\rightarrow \mathfrak{aut}(G'/H, (J_t)_|)$, and take the average of $\omega_K$ by the image of the closure of $\t_t$ (we still call it $\omega_K$).
    Therefore, for any vector $X\in\t_t$, $\rho(X)$ is $(J_t)_|$-holomorphic and $\omega_K$-Killing. Since $\t_t$ is $J_t$-closed, the same is true for $(J_t)_| \rho(X)=\rho(J_tX)$ as well. It follows that $\nabla^{g_K}\rho(X)=0$ where $\nabla^{g_K}$ is the Levi--Civita connection of $\omega_K$.
    However, there are no non-trivial parallel vector fields on the flag manifolds, hence $\rho(X)=0$.
    Since the Kernel of the projection $\rho$ is $\ker(\rho)=\h^L$, and the dimensions match, it holds $\t_t=\h^L$. 
    In other words $\t_t\subset \g'$ is the maximal torus of the Samelson construction, and the orbit of $\t_t$ is the maximal torus $H\subset G$.
    However, all the $\t_t$ with irrational slope have non-closed orbits dense in tori of dimension grater than $r$, giving the contradiction that proves (i).
    \medskip

    {\bf Step 3:} {\em Detecting a transversely K\"ahler metric.}\\
    The horizontal distribution $\H$ gives a smooth family of principal connections $\mu_t\in\bigwedge^1(M,\t_t)$ for the abelian foliations $\F_t$.
    Their curvatures $\Omega_t$ are closed basic $(1,1)$-forms $\Omega_t\in \A_b^{1,1}(M,\F_t,\t_t)$.
    Indeed,
    $$\Omega_t\left(\H^{1,0},\H^{1,0}\right) = \mu_t\left(\left[\H^{1,0},\H^{1,0}\right]\right) = 0.$$
    Let us fix coroots $V_1,\ldots,V_r$ generating $\t$ associated to $J$ as in Section \ref{subs: flat and flag}. We extend these to a family $V_{t,1},\ldots,V_{t,r}$ varying smoothly on $t$ and generating $\t_t$, and decompose the curvature form accordingly, 
    $$\Omega_t = \sum_{j=1}^rV_{t,j}\,\otimes \Omega_{t,j} \quad \text{with }\quad \Omega_{t,j}\in \A_b^{1,1}(M,\F_t) \quad \text{for } j=1,\ldots,r.$$
    For $t=0$ the curvature form gives a K\"ahler form on $G/T$ as $\Omega_{0,1}+\cdots+\Omega_{0,r}$ thanks to \eqref{eq: flag kahler}. 
    Then, for small $t$, $\Omega_{t,1}+\cdots+\Omega_{t,r}$ is non-degenerate, hence, gives a transversely K\"ahler structure on $(\F_t,\H)$.
    \medskip
    
    {\bf Step 4:} {\em Computing the basic de Rham cohomology.}\\
    We shall consider here any simple component $G^{(i)}$ on its own. So let us relabel the vector fields generating $\t_t$ and the connection and curvature forms accordingly as
    \[
        V_{t,j}^{(i)}, \; \mu_{t,j}^{(i)},\; \text{and } \Omega_{t,j}^{(i)}, \quad \text{for } i=1,\ldots,s, \;\text{and } j=1,\ldots,r_i.
    \]
    For any $i=1,\ldots,s$ we have a family of foliations $\F_t^{(i)}=\big(G^{(i)},\t_t^{(i)}\big)$ which are transversely K\"ahler with the horizontal distribution $\H^{(i)}\defeq\H\cap TG^{(i)}$, the restriction of $J$ to it (which coincide with the complex structure induced on $G^{(i)}/T^{(i)}$), and the transversal metric $\Omega_{t,1}^{(i)}+\cdots+\Omega_{t,r_i}^{(i)}$.
    Thus, \cite[Theorem 4.13]{MR3918619} guaranties the existence of a model for the de Rham cohomology of $G^{(i)}$ induced by the foliation structure.
    Defining $\W_t^{(i)}$ as the space generated by the connection forms $\mu_{t,j}^{(i)}$, the model is given by
    \begin{equation}\tag{dR-model}\label{dR-model}
        \Psi_{dR}:\left(H^\bullet_b\left(G^{(i)},\F_t^{(i)}\right)\otimes\bigwedge\W_t^{(i)},d'\right) \rightarrow \left(\A^\bullet G^{(i)},d\right) \,,
    \end{equation}
    that means that $\Psi_{dR}$ induces an isomorphism in cohomology.
    Notice that $d'\equiv0$ on $H^\bullet_b\big(G^{(i)},\F_t^{(i)}\big)$, and $d'\equiv d$ on $\W_t^{(i)}$ since $d\W_t^{(i)}$ is naturally embedded in $H^2_b\big(G^{(i)},\F_t^{(i)}\big)$.
    Recall that the Betti numbers in low degrees are the same for any compact simply connected simple Lie group, and are $b_0=1,\, b_1=0,\, b_2=0,\, b_3=1,\, b_4=0$.
    A first easy consequence, coming from $b_1=0$, is that $H^1_b\big(G^{(i)},\F_t^{(i)}\big)=0$.
    Moreover, $b_1=0$ also implies that $d'\mu^{(i)}_{t,j}\neq0$ in \eqref{dR-model}, that means $\Omega^{(i)}_{t,j}\in H^2_b\big(G^{(i)},\F_t^{(i)}\big)$ for all $j=1,\ldots,r_i$.
    On the other hand, $b_2=0$ implies that for any $\alpha\in H_b^2\big(G^{(i)},\F_t^{(i)}\big)$ it holds $\alpha=d'\beta$ for some $\beta\in\W_t^{(i)}$. Hence, $H^2_b\big(G^{(i)},\F_t^{(i)}\big)$ is precisely generated by $\Omega^{(i)}_{t,1},\ldots,\Omega^{(i)}_{t,r_i}$.
    Now, in order to realize $b_3=1$, one of the following has to be true:
    \begin{itemize}
        \item[(1)]  There exists $\eta\in H^3_b\big(G^{(i)},\F_t^{(i)}\big)$ which is not $d'$-exact,
        \item[(2)] There exists $\alpha\in H^2_b\big(G^{(i)},\F_t^{(i)}\big)\otimes\W_t^{(i)}$ such that $d'\alpha=0$.
    \end{itemize}
    Let us suppose (1) holds.
    Since the foliation is homologically oriented by \cite[Proposition 4.8]{MR3918619} and transversely K\"ahler, the basic cohomology satisfies the Hodge decomposition \cite{MR1042454}
    $$\overline{H^{p,q}_b}\left(G^{(i)},\F_t^{(i)}\right) = H^{q,p}_b\left(G^{(i)},\F_t^{(i)}\right), \quad\text{and }\quad H_b^k\left(G^{(i)},\F_t^{(i)},\C\right)=\bigoplus_{p+q=k}H_b^{p,q}\left(G^{(i)},\F_t^{(i)}\right).$$ 
    Therefore, there actually exist two different elements $\eta_1,\eta_2\in H^3_b\big(G^{(i)},\F_t^{(i)}\big)$, and $b_3=1$ forces the existence of some $\gamma\in H_b^1\big(G^{(i)},\F_t^{(i)}\big)\otimes\W_t^{(i)}$ such that $d'\gamma$ is a non-trivial linear combination of $\eta_1,\eta_2$. However, we know $H_b^1\big(G^{(i)},\F_t^{(i)}\big)=0$.
    Hence, (2) holds.
    Using the fact that $H_b^2\big(G^{(i)},\F_t^{(i)}\big)=\C\left<\Omega^{(i)}_{t,1},\ldots,\Omega^{(i)}_{t,r_i}\right>$,we know that
    $$\alpha = \sum_{j,k=1}^{r_i}a_{j,k}\,\mu^{(i)}_{t,j}\otimes \Omega^{(i)}_{t,k},
    \quad \text{and }\quad d\alpha = \sum_{j,k=1}^{r_i}a_{j,k} \,\Omega^{(i)}_{t,j}\otimes \Omega^{(i)}_{t,k}, \quad \text{for }a_{j,k}\in\C.$$
    Therefore, (2) implies the existence of a symmetric quadratic form $Q_t^{(i)}$ such that $Q_t^{(i)}(\Omega^{(i)}_{t,k},\Omega^{(i)}_{t,j})$ vanishes in $H^4_b\big(G^{(i)},\F_t^{(i)}\big)$, or in other words, it is $d$-exact from a basic form.
    Now suppose that two such polynomials exist, $Q_t^{(i)},\widetilde{Q}_t^{(i)}$. 
    We can find two forms $\sigma,\widetilde{\sigma}\in H_b^2\big(G^{(i)},\F_t^{(i)}\big)\otimes\W_t^{(i)}$ such that $d\sigma=Q^{(i)}_t(\Omega^{(i)}_{t,k},\Omega^{(i)}_{t,j})$ and $d\widetilde\sigma=\widetilde Q^{(i)}_t(\Omega^{(i)}_{t,k},\Omega^{(i)}_{t,j})$, that implies $d'\sigma=d'\widetilde{\sigma}=0$.
    Since $b_3=1$ and $H_b^1\big(G^{(i)},\F_t^{(i)}\big)=0$, there exists $\gamma\in\W_t^{(i)}\bigwedge\W_t^{(i)}$ such that $d'\gamma=c_1\sigma+c_2\widetilde\sigma$ for some constants $c_1,c_2$ not simultaneously vanishing. 
    But this implies that $c_1Q^{(i)}_t(\Omega^{(i)}_{t,k},\Omega^{(i)}_{t,j})+c_2\widetilde Q^{(i)}_t(\Omega^{(i)}_{t,k},\Omega^{(i)}_{t,j}) = c_1d\sigma+c_2d\widetilde\sigma=dd'\gamma=dd\gamma=0$. That is $Q^{(i)}_t$ and $\widetilde Q^{(i)}_t$ are multiple.
    In other words, any such a quadratic form gives a representative in $H^3(G^{(i)})$, so there is one and only one of them (up to multiplicative constants).
    To sum up, any quadratic form in the $\Omega^{(i)}_{t,j}$ but one gives a representative in $H^4_b\big(G^{(i)},\F_t^{(i)}\big)$.
    Furthermore, there are no other representatives in $H^4_b\big(G^{(i)},\F_t^{(i)}\big)$ besides these, indeed they could not be in the image of $d'$, hence giving a representative in $H^4(G^{(i)})$ which we know is trivial.
    In formulas, we have
    $$H^4_b\big(G^{(i)},\F_t^{(i)}\big) = \R\left< \Omega^{(i)}_{t,k} \wedge \Omega^{(i)}_{t,j} \right>_{k,j=1,\ldots,r_i} / Q^{(i)}_t(\Omega^{(i)}_{t,k},\Omega^{(i)}_{t,j}).$$
    Considering all the simple components we get
    $$H^4_b(G,\F_t) = \R\left< \Omega_{t,k} \wedge \Omega_{t,j} \right>_{k,j=1,\ldots,r} / \left< Q^{(i)}_t(\Omega^{(i)}_{t,k},\Omega^{(i)}_{t,j}) \right>_{i=1,\ldots,s} .$$
    In particular, its dimension is $\sum_{i=1}^s (r_i(r_i+1)/2 - 1) $, independent on $t$. 
    
    \medskip
    
    {\bf Step 5:} {\em Proving the basic harmonic 4-forms have a smooth bundle structure.}\\
    Consider the elliptic self-adjoint operators $D_t:\A_\H^4(G,\F_t)\rightarrow\A_\H^4(G,\F_t)$ as defined in \eqref{eq: new Laplacian}, and notice that both the space of orizontal forms $\A_H^4(G,\F_t)$ and the operator $D_t$ vary smoothly on $t$.
    By standard theory (see \cite{MR115189}) there is a $\varepsilon>0$ and a neighborhood $U$ of $0\in\R^\ell$ such that for $t\in U$
    $$\Upsilon_t \defeq \mathrm{span}\left\{ \xi\in\A_\H^4(G,\F_t) \;|\; D_t\xi = \lambda\xi, \; |\lambda|<\varepsilon \right\}$$
    forms a smooth bundle over $U$ (see also \cite[Theorem 3.4]{MR4278212} for more details).
    Now consider the operators $\L_{V_{t,i}}:\Upsilon_t\rightarrow \A_T^4(G,\F_t)$ for $i=1,\ldots,r$.
    Up to replacing $U$ with a smaller neighborhood we have that for any $t\in U$
    $$\dim \left(\bigcap_{i=1}^r\ker \L_{V_{0,i}} \right)\geq\dim \left(\bigcap_{i=1}^r\ker \L_{V_{t,i}} \right).$$
    As explained in Section \ref{subs: new Laplacian}, we can recover the harmonic basic forms as $\mathbf{H}_b^4(G,\F_0) = \bigcap_{i=1}^r\ker \L_{V_{0,i}}$, while evidently $\mathbf{H}_b^4(M,\F_t) \subset \bigcap_{i=1}^r\ker \L_{V_{t,i}}$.
    Therefore, since $\dim(\mathbf{H}_b^4(M,\F_t))$ does not depend on $t$, we have
    $$\dim\left(\mathbf{H}_b^4(M,\F_0)\right)=\dim\left(\mathbf{H}_b^4(M,\F_t)\right)\leq \dim \left(\bigcap_{i=1}^r\ker \L_{V_{t,i}} \right) \leq \dim \left(\bigcap_{i=1}^r\ker \L_{V_{0,i}} \right) = \dim\left(\mathbf{H}_b^4(M,\F_0)\right) . $$
    Thus, also $\dim \left(\bigcap_{i=1}^r\ker \L_{V_{t,i}} \right)$ does not depend on $t$.
    This means that these operators cut a smooth bundle from $\A_\H^4(G,\F_t)$, which is precisely that of the harmonic forms $\mathbf{H}_b^4(M,\F_t)$.
    \medskip

    {\bf Step 6:} {\em Defining the complex structure from the quadratic form.}\\
    The fact that the space of harmonic basic $4$-forms varies smoothly with $t$ implies that also the quadratic forms $Q_t^{(i)}$ depend smoothly on $t$.
    As a matter of fact, the projections on $\mathbf{H}_b^4(M,\F_t)$ are smooth in $t$, hence the harmonic components of the forms $\Omega_{t,j}^{(i)}\wedge\Omega_{t,k}^{(i)}$ vary smoothly on $t$, and $Q_t^{(i)}$ represents the non-trivial linear combination of these that vanishes.

    We can rescale the vectors $V_{0,1},\ldots,V_{0,r}$ to get a $g$-orthonormal basis (notice that they are already orthogonal).
    Then $J_{|_\t}$ is an orthogonal matrix in that basis.
    We know from \eqref{eq: vanishing Q-form} that $Q_0^{(i)}=\id$.
    Thus $Q_t^{(i)}$ is a non-degenerate quadratic form for $t$ small enough.
    Define $Q_t = \sum_{i=1}^s\lambda_{t,i}Q_t^{(i)}$ with $\lambda_{t,1},\ldots,\lambda_{t,s}\in\R$ smooth in $t$ and such that $\lambda_{0,1}=\cdots=\lambda_{0,s}=1$.
    This is a non-degenerate quadratic form for small $t$, hence there exists a linear complex structure $J_t$ varying smoothly on $t$ such that $Q_t=J_tQ_tJ_t^T$ and $J_0=J_{|_\t}$.
    We use this matrix to define a complex structure on $\t$. 
    Then, as explained in Step 1, we get a complex structure on $G$, which we still indicate with $J_t$, that vary smoothly on $t$ and such that $J_0=J$.
    Notice that in this step we have enlarged the deformation space with the degrees of freedom given by the coefficients $\lambda_{t,i}$ and the choice of a $Q_t$-orthogonal linear complex structure.
    \medskip

    {\bf Step 7:} {\em Proving the $(1,1)$-Aeppli number is constant.}\\
    We use the complex structure just defined to decompose $\W_t\defeq\W_t^{(1)}+\cdots+\W_t^{(s)}\subset\A^1(G)$ in $(1,0)$-part $\W_t^{1,0}$ and $(0,1)$-part $\W_t^{0,1}$.
    Using again \cite[Theorem 4.13]{MR3918619}, the foliation structure induces a quasi isomorphism of double graded algebras given by
    \begin{equation}\tag{Aeppli-model}\label{a-model}
        \Psi_{A} : \left( H^{\bullet,\bullet}_b(G,\F_t) \otimes \bigwedge(\W_t^{1,0}\oplus\W_t^{0,1}), \p',\bar\p'\right) \rightarrow \left(\A^{\bullet,\bullet}(G,J_t),\p,\bar\p\right) \,;
    \end{equation}
    in particular $\Psi_A$ induces an isomorphisms in the Aeppli cohomology. 
    Notice that $\p'\equiv 0$ on $H^{\bullet,\bullet}_b(G,\F_t)$ and $\p' \equiv \p$ on $\W_t^{1,0}\oplus\W_t^{0,1}$ since $\p(\W_t^{1,0}\oplus\W_t^{0,1})$ is naturally embedded in $H^{1,1}_b(G,\F_t)$; similarly for $\bar\p'$.
    Now take an element 
    $$ \eta_A = \sum_{j,k=1}^rA_{j,k}(\mu_{t,j}-\im J_t\mu_{t,j}) \wedge (\mu_{t,k}+\im J_t\mu_{t,k}) \in \W_t^{1,0}\wedge\W_t^{0,1}$$
    defined by a $\C$-valued $(r\times r)$-matrix $A$. It is not exact in the Aeppli cohomology sense, and moreover it holds
    \begin{align*}
        \p'\bar\p' \eta_A
            &= \sum_{j,k,p,q=1}^rA_{j,k} \left( (\bar\p\mu_{t,j}-\im (J_t)^j_p\bar\p\mu_{t,p}) \wedge (\p\mu_{t,k}+\im (J_t)^k_q\p\mu_{t,q}) \right) \\
            &= \sum_{j,k,p,q=1}^rA_{j,k} \left( (d\mu_{t,j}-\im (J_t)^j_p d\mu_{t,p}) \wedge (d\mu_{t,k}+\im (J_t)^k_q d\mu_{t,q}) \right) \\
            &= \sum_{j,k,p,q=1}^rA_{j,k} \left( \Omega_{t,j}\wedge\Omega_{t,k} -\im (J_t)^j_p\Omega_{t,p}\wedge\Omega_{t,k} +\im (J_t)^k_q\Omega_{t,j}\wedge\Omega_{t,q} + (J_t)^j_p(J_t)^k_q \Omega_{t,p}\wedge\Omega_{t,q}\right) \\
            &= \sum_{j,k,p,q=1}^r \left( A_{j,k} -\im (J_t)_j^pA_{p,k} +\im (J_t)_k^qA_{j,q} + (J_t)_j^p(J_t)_k^qA_{p,q} \right) \Omega_{t,j}\wedge\Omega_{t,k} \\
            &= \sum_{j,k=1}^r \left( A -\im J_tA +\im AJ_t^T + J_tAJ_t^T \right)_{j,k}\, \Omega_{t,j}\wedge\Omega_{t,k}
    \end{align*}
    where we used twice that $\p(\mu_{t,j}-\im J_t\mu_{t,j}) = \p(\mu_{t,j})^{1,0}=(\Omega_{t,j})^{2,0}=0$ and similarly $\bar\p(\mu_{t,j}+\im J_t\mu_{t,j})=0$.
    Taking $A=\tfrac{1}{2}Q_t$ we obtain $\p'\bar\p'\eta_A = Q_t(\Omega_{t,j},\Omega_{t,k})$ which is zero in the basic cohomology $H_b^{2,2}(G,\F_t) \cong H_b^4(G,\F_t,\C)$ by the transversal $\p\bar\p$-Lemma. 
    Therefore, these give representatives in the Aeppli cohomology of $(G,J_t)$ which then satisfies $h^{1,1}_A(G,J_t)\geq 1$.
    Notice that we know that $h^{1,1}_A(G,J) = 1$, hence, by upper-semi-continuity, $h^{1,1}_A(G,J_t) = 1$ for small $t$.
\end{proof}

We highlight that, since the holomorphic structure induced on the flag manifold $G/T$ is rigid, the only non-trivial infinitesimal deformations of the transverse holomorphic structure of $G$ are of the form considered in Theorem \ref{th: main}.
As a matter of fact, it was proved in \cite{MR594704, MR614365} that for $(\F,J)$ a holomorphic foliation on $G$, the infinitesimal deformations of $(\F,J)$ correspond to elements of $H^1(G,\Theta_\F)$, where $\Theta_\F$ is the sheaf of germs of holomorphic vector fields on the normal bundle of $\F$ which are constant on the leaves of $\F$.
We can use Leray's spectral sequence for the Tits fibration \eqref{eq: Tits} to compute this cohomology, adapting the classical argument used for Sasakian structures (see e.g. \cite{MR2382957}).
To do so, we have to consider the Leray sheaf $R^q\pi_*\Theta_\F$ associated to $\pi$, which for any integer $q$ is the sheafification of the presheaf 
$$G/T \,\stackrel{open}{\supset}\, U\mapsto H^q\left(\pi^{-1}(U),\left(\Theta_\F\right)_{|_{\pi^{-1}(U)}}\right).$$
Notice that the stalk in $x\in G/T$ of this presheaf satisfies
$$H^q\left(\pi^{-1}(x),\left(\Theta_\F\right)_{|_{\pi^{-1}(x)}}\right) = H^q\left(T_x,\left(\Theta_\F\right)_{|_{T_x}}\right) = H^q\left(T_x,\left(\Theta_{G/T}\right)_{x}\right) = 
\begin{cases}
    \left(\Theta_{G/T}\right)_{x}^{\binom{r}{q}} & \text{if } 0\leq q\leq r;\\
    0 & \text{otherwise,}
\end{cases}$$
where $T_x$ is the $T$-orbit of $x$, $\Theta_{G/T}$ is the sheaf of germs of holomorphic vector fields on $G/T$, and we used the fact that the vector fields in $\Theta_\F$ are constant along the leaves of $\F$.
Then the Leray sheaf is
$$R^q\pi_*\Theta_\F = \begin{cases}
    \left(\Theta_{G/T}\right)^{\binom{r}{q}} & \text{if } 0\leq q\leq r;\\
    0 & \text{otherwise.}
\end{cases} $$
By Leray's theorem the spectral sequence given by $E_2^{p,q}= H^p(G/T,R^q\pi_*\Theta_\F)$ converges to $H^\bullet(G,\Theta_\F)$.
Furthermore, a theorem of Bott ensures that $H^p(G/T,\Theta_{G/T})=0$ for $p>0$ \cite{MR89473}, that is precisely saying that the holomorphic structure on the flag is rigid.
Therefore, the spectral sequence degenerates at the second page since it has only one non-trivial row.
The convergence implies the existence of an exact sequence
$$0\rightarrow E_\infty^{1,0}\rightarrow H^1(G,\Theta_\F) \rightarrow E_\infty^{0,1}\rightarrow 0,$$
and thanks to the above argument $E_\infty^{1,0} \cong E_2^{1,0}=0$ and $H^1(G,\Theta_\F)\cong E_2^{0,1} = H^0\left(G/T, \left(\Theta_{G/T}\right)^r\right)$.
This is precisely saying that the non-trivial infinitesimal deformations of the transversely holomorphic structure on $G$ are given by a choice of $r$ holomorphic vector fields on $G/T$. In other words, they come from deformations of the vector fields generating $\t=\Lie{T}$ which arise from the infinitesimal complex automorphisms of $G/T$.

\bibliographystyle{alpha}
\bibliography{Bibliography}

\end{document}